\documentclass[12pt]{article}

\usepackage{amsmath,amssymb,amsthm,amssymb}
\usepackage{tikz,tkz-graph,tikz-cd,tikz-qtree}
\usepackage{youngtab}
\usepackage[OT2,T1]{fontenc}
\usepackage{enumitem}

\DeclareSymbolFont{cyrletters}{OT2}{wncyr}{m}{n}

\DeclareMathSymbol{\Sha}{\mathalpha}{cyrletters}{"58}

\newtheorem{theorem}{Theorem}[section]
\newtheorem{lemma}{Lemma}[section]
\newtheorem{corollary}{Corollary}[section]

\newtheorem{definition}{Definition}[section]
\newtheorem{proposition}{Proposition}[section]
\newtheorem{remark}{Remark}[section]

\newtheorem{examples}{Example}[section]

\begin{document}

\bibliographystyle{abbrv}
	
\title{The Multiple Equal-Difference Structure of Cyclotomic Cosets}
	
\author{Li Zhu$^{1}$, Juncheng Zhou$^{2}$, Jinle Liu$^{2}$ and Hongfeng Wu$^{2}$\footnote{Corresponding author.}
\setcounter{footnote}{-1}
\footnote{E-Mail addresses:
lizhumath@pku.edu.cn(L. Zhu), 2024312050101@mail.ncut.edu.cn(J. Zhou), cohomologyliu@163.com(J. Liu), whfmath@gmail.com(H. Wu)}
\\
{1.~School of Mathematical Sciences, Guizhou Normal University, Guiyang, China}
\\
{2.~College of Science, North China University of technology, Beijing, China}}
	
\date{}
\maketitle
	
\thispagestyle{plain}
\setcounter{page}{1}	
	
\begin{abstract}
	In this paper we introduce the definition of equal-difference cyclotomic coset, and prove that in general any cyclotomic coset can be decomposed into a disjoint union of equal-difference subsets. Among the equal-difference decompositions of a cyclotomic coset, an important class consists of those in the form of cyclotomic decompositions, called the multiple equal-difference representations of the coset. There is an equivalent correspondence between the multiple equal-difference representations of $q$-cyclotomic cosets modulo $n$ and the irreducible factorizations of $X^{n}-1$ in binomial form over finite extension fields of $\mathbb{F}_{q}$. We give an explicit characterization of the multiple equal-difference representations of any $q$-cyclotomic coset modulo $n$, through which a criterion for $X^{n}-1$ factoring into irreducible binomials is obtained. In addition, we present an algorithm to simplify the computation of the leaders of cyclotomic cosets.\\


\end{abstract}
	
\section{Introduction}
Cyclotomic coset is a classical notion in the theory of finite fields. As they are closely related to the factorizations of polynomials over finite fields, cyclotomic cosets are widely involved in the problems from, for instance, finite fields, coding theory, cryptography and computational number theory. In particular, cyclotomic cosets play an important role in the theory of constacyclic codes. In practice, one often need to compute concretely the parameters associated to cyclotomic cosets, such as the representatives, leaders, sizes and enumerations of the cosets.

Many results have been achieved under specific conditions, some of which we list as follows. Being restricted by the authors' knowledge, the list is limited. In a series of papers, for instance \cite{Chen}, \cite{Chen 2}, \cite{Sharma}, \cite{Liu}, \cite{Sharma 2}, \cite{Liu 2}, \cite{Wu-Yue}, \cite{Wu 2}, etc., the representatives and sizes of the corresponding classes of cyclotomic cosets are determined in order to calculate certain constacyclic codes. In \cite{Chen 3} and \cite{Wang} the $q$-cyclotomic cosets contained in the subset $1+r\mathbb{Z}/nr\mathbb{Z}$ of $\mathbb{Z}/nr\mathbb{Z}$, where $\mathrm{gcd}(q,nr)=1$ and $r \mid q-1$, are characterized and enumerated. Applying the results on cyclotomic cosets, \cite{Chen 3} gives the enumeration of Euclidean self-dual codes, and \cite{Wang} construct several classes of $p^{h}$-LCD MDS codes. And in \cite{Yue} and \cite{Geraci}, concerning with stream cipher $m$-sequences and problems in statistic physics respectively, algorithms to calculate the leader of cyclotomic cosets are given.

This paper is devoted to introduce the definition of equal-difference cyclotomic coset and characterize the multiple equal-difference structure of a general coset. Intuitively, a $q$-cyclotomic coset modulo $n$ is of equal difference if it can be presented as a complete arithmetic sequence lying in $\{0,1,\cdots,n-1\}$. In Section \ref{sec 2} we investigate the basic properties of equal-difference cyclotomic cosets; particularly, the criterion for a coset to be equal difference, and that for all $q$-cyclotomic cosets modulo $n$ to be of equal difference are given respectively.

In general, a cyclotomic coset $c_{n/q}(\gamma)$ is not necessarily of equal difference, however, it turns out that $c_{n/q}(\gamma)$ can be expressed as a disjoint union of equal-difference subsets. Among the equal-difference decompositions of $c_{n/q}(\gamma)$, an important class, called the multiple equal-difference representations of $c_{n/q}(\gamma)$, come from the $q^{t}$-cyclotomic decomposition for $t \in \mathbb{N}^{+}$, that is, they are of the form
$$c_{n/q}(\gamma) = \bigsqcup_{j=0}^{\mathrm{gcd}(t,\tau)-1}c_{n/q^{t}}(\gamma q^{j}).$$
In Section \ref{sec 3}, we determine all the equal-difference decompositions of any cyclotomic coset, and give an explicit characterization to the class of multiple equal-difference representations.

For any prime power $q$ and positive integer $n$ coprime to $q$, the $q$-cyclotomic cosets modulo $n$ can be seen as the Frobenius orbits of the $n$-th roots of unity over $\mathbb{F}_{q}$, therefore they correspond exactly to the irreducible factors of $X^{n}-1$ over $\mathbb{F}_{q}$. In Section \ref{sec 4} we prove that a coset $c_{n/q}(\gamma)$ is of equal difference if and only if the corresponding polynomial $M_{c_{n/q}(\gamma)}(X)$ is a binomial. Moreover, in general, the multiple equal-difference representations of a cyclotomic coset $c_{n/q}(\gamma)$ are in an anti-order-preserving one-to-one correspondence with the irreducible factorizations of $M_{c_{n/q}(\gamma)}(X)$ in binomial form over finite extension fields of $\mathbb{F}_{q}$. In this way, the multiple equal-difference structure of cyclotomic cosets encodes significant information on the factorizations of $X^{n}-1$. Based on the properties of equal-difference cosets, we prove a criterion on $q$ and $n$ for $X^{n}-1$ factoring into irreducible binomials over $\mathbb{F}_{q}$.

In Section \ref{sec 5}, as an application we present an algorithm that simplify the computation of the leaders of cyclotomic cosets. In particular, the leader of any equal-difference coset is determined. We expect that this algorithm can be applied in the problems from coding theory and cryptography.

\section{Preliminaries}\label{sec 1}
In this section we fix the basic notations and recall some facts that are needed in the following context.
\subsection{Basic number theory}
Let $n$ be a positive integer with the prime decomposition $n = p_{1}^{e_{1}}\cdots p_{s}^{e_{s}}$, where $p_{1},\cdots,p_{s}$ are distinct primes, and $e_{1},\cdots,e_{s}$ are positive integers. The radical of $n$ is defined by $\mathrm{rad}(n) = p_{1}\cdots p_{s}$.

If $m$ and $n$ are coprime integers, we denote by $\mathrm{ord}_{n}(m)$ the order of $m$ in the multiplicative group $(\mathbb{Z}/n\mathbb{Z})^{\ast}$, i.e., the smallest positive integer such that 
$$m^{\mathrm{ord}_{n}(m)} \equiv 1 \pmod{n}.$$
It is clear that $\mathrm{ord}_{n}(m)$ divides the order $\phi(n)$ of $(\mathbb{Z}/n\mathbb{Z})^{\ast}$.

Let $\ell$ be a prime. Denote by $v_{\ell}(n)$ the $\ell$-adic valuation of $n$, i.e., the maximal integer such that $\ell^{v_{\ell}(n)} \mid n$. The following lift-the-exponent lemmas are well-known.
	
\begin{lemma}{\cite{Nezami}}\label{lem 3}
	Let $\ell$ be an odd prime number, and $m$ be an integer such that $\ell \mid m-1$. Then $v_{\ell}(m^{d}-1) = v_{\ell}(m-1) + v_{\ell}(d)$ for any positive integer $d$.
\end{lemma}

\begin{lemma}{\cite{Nezami}}\label{lem 2}
	Let $m$ be an odd integer, and $d$ be a positive integer.
	\begin{description}
		\item[(1)] If $m \equiv 1 \pmod{4}$, then
		$$v_{2}(m^{d}-1) = v_{2}(m-1) + v_{2}(d), \  v_{2}(m^{d}+1) = 1.$$
		\item[(2)] If $m \equiv 3 \pmod{4}$ and $d$ is odd, then
		$$v_{2}(m^{d}-1) = 1, \  v_{2}(m^{d}+1) = v_{2}(m+1).$$
		\item[(3)] If $m \equiv 3 \pmod{4}$ and $d$ is even, then
		$$v_{2}(m^{d}-1) = v_{2}(m+1) + v_{2}(d), \  v_{2}(m^{d}+1) = 1.$$
	\end{description}
\end{lemma}

\subsection{Finite fields}
Let $\mathbb{F}_{q}$ be a finite field with $q$ elements, where $q$ is a prime power. The multiplicative group $\mathbb{F}_{q}^{\ast}$ of nonzero elements is a cyclic group. For any $\lambda \in \mathbb{F}_{q}^{\ast}$, the smallest positive integer $r$ such that $\lambda^{r} = 1$ is called the order of $\lambda$ and is denoted by $r = \mathrm{ord}(\lambda)$. It is obvious that $\mathrm{ord}(\lambda)$ is a divisor of $q-1$. 

Let $f$ be an irreducible polynomial over $\mathbb{F}_{q}$ which is not equal to $aX$ for any $a \neq 0$. The order $\mathrm{ord}(f)$ of $f$ is defined to be the smallest positive integer $r$ such that $f \mid X^{r}-1$. One can verify that $\mathrm{ord}(f) = \mathrm{ord}(\alpha)$ for any root $\alpha$ of $f$ lying in the finite extension field $\mathbb{F}_{q}[X]/(f)$ of $\mathbb{F}_{q}$.

For any positive integer $n$ coprime to $q$, there are $n$ roots of $X^{n}-1$, lying in some extension field of $\mathbb{F}_{q}$, that form a cyclic group $\mu_{n}$. Any generator of $\mu_{n}$ is called a primitive $n$-th root of unity. In this paper, we fix a compatible family
$$\zeta = \{\zeta_{n} \ | \ \mathrm{gcd}(n,q) = 1\}$$
of primitive roots of unity, where the compatibility means that for any positive integer $m$ and $n$ such that $\mathrm{gcd}(mn,q)=1$ and $m \mid n$, it holds that $\zeta_{n}^{\frac{n}{m}} = \zeta_{m}$.

Let $n$ be a positive integer coprime to the prime power $q$. Given any $\gamma \in \mathbb{Z}/n\mathbb{Z}$, the $q$-cyclotomic coset modulo $n$ containing $\gamma$ is defined to be 
$$c_{n/q}(\gamma) = \{\gamma,\gamma q,\cdots,\gamma q^{\tau-1}\} \subseteq \mathbb{Z}/n\mathbb{Z},$$
where $\tau$, called the size of $c_{n/q}(\gamma)$, is the smallest positive integer such that $\gamma q^{\tau}\equiv\gamma\pmod n$. Any element in $c_{n/q}(\gamma)$ is called a representative of the coset $c_{n/q}(\gamma)$. When it makes no confusion, we often do not distinguish an element in $\mathbb{Z}/n\mathbb{Z}$ with its any primage in $\mathbb{Z}$. In particular, if viewing each representative of $c_{n/q}(\gamma)$ as a nonnegative integer less than $n$, the smallest one is called the leader of $c_{n/q}(\gamma)$. Moreover, we denote the space of all $q$-cyclotomic cosets modulo $n$ by $\mathcal{C}_{n/q}$.

It is well-known that the $q$-cyclotomic cosets modulo $n$ fully determine the irreducible factorization of $X^{n}-1$ over $\mathbb{F}_{q}$, which, however, depends on the choice of a primitive $n$-th root of unity. In fact, a different choice gives rise to a permutation of the irreducible factors of $X^{n}-1$. Throughout this paper, we always choose the $n$-th primitive root $\zeta_{n}$ lying in $\zeta$ so that any $q$-cyclotomic coset $c_{n/q}(\gamma) = \{\gamma,\gamma q,\cdots,\gamma q^{\tau-1}\}$ induces the irreducible factor 
$$M_{c_{n/q}(\gamma)}(X) = (X-\zeta_{n}^{\gamma})(X-\zeta_{n}^{\gamma q})\cdots(X-\zeta_{n}^{\gamma q^{\tau-1}})$$
of $X^{n}-1$ over $\mathbb{F}_{q}$.

\section{Equal-difference cyclotomic cosets}\label{sec 2}
Let $q$ be a power of a prime $p$, and $n$ be a positive integer not divisible by $p$. This section is devoted to define equal-difference cyclotomic coset and investigate the basic property. In particular, we give an equivalent characterization of the equal-difference cyclotomic cosets, and also a criterion on $q$ and $n$ for every $q$-cyclotomic coset modulo $n$ being of equal difference.

\begin{definition}\label{def 1}
	Let $\gamma \in \mathbb{Z}/n\mathbb{Z}$, and $c_{n/q}(\gamma)$ be the associated $q$-cyclotomic coset modulo $n$, with $|c_{n/q}(\gamma)| = \tau$. The coset $c_{n/q}(\gamma)$ is called an equal-difference cyclotomic coset if $\tau \mid n$ and $c_{n/q}(\gamma)$ can be presented as
	$$c_{n/q}(\gamma) = \{\gamma, \gamma+\dfrac{n}{\tau}, \cdots, \gamma+(\tau-1)\dfrac{n}{\tau}\}.$$
	The quotient $\frac{n}{\tau}$ is called the common difference of $c_{n/q}(\gamma)$. In particular, any coset containing only one element is regarded as an equal-difference coset.
\end{definition}
	
\begin{remark}
	We require the condition $\tau \mid n$ in Definition \ref{def 1} so that besides the multiplicative coset structure, an equal-difference cyclotomic coset is also equipped with a structure of a coset with respect to an additive subgroup of $\mathbb{Z}/n\mathbb{Z}$. It turns out that this condition is well-imposed as it is necessary for many properties and the multiple equal-difference structure of cyclotomic coset in general to be valid.
\end{remark}

Intuitively, an euqal-difference coset is exactly a coset which can be represented as a complete arithmetic sequence contained in the set $\{0,1,\cdots,n-1\}$, that is, there is a positive integer $d$ with which $c_{n/q}(\gamma)$ can be written as
\begin{equation}\label{eq 1}
	c_{n/q}(\gamma) = \{\gamma_{0},\gamma_{0}+d,\cdots,\gamma_{0}+(\tau-1)d\} \subseteq \{0,1,\cdots,n-1\},
\end{equation}
where $\gamma_{0}$ is the leader of $c_{n/q}(\gamma)$, and the completeness means that $\gamma_{0} + \tau d \equiv \gamma_{0} \pmod{n}$. The completeness condition implies that $c_{n/q}(\gamma)$ can be extended to an arithmetic sequence contained in $\mathbb{Z}$.

It is clear from Definition \ref{def 1} that an equal-difference coset can be written in the form \eqref{eq 1}. Conversely, let $c_{n/q}(\gamma)$ be any coset in the form \eqref{eq 1} that satisfies $\gamma_{0} + \tau d \equiv \gamma_{0} \pmod{n}$. If $\tau = 1$, the conclusion is trivial. If $\tau > 1$, noting that $n \mid \tau d$ and $0 < (\tau-1)d < n$, then we have $d = \frac{n}{\tau}$.
 
We exhibit some examples of equal-difference cyclotomic cosets as follows.

\begin{examples}
	\begin{description}
		\item[(1)] Let $q = 5$ and $n=32$. All $q$-cyclotomic cosets modulo $n$ are 
		\begin{align*}
			&\{0\}, \ \{1,5,9,13,17,21,25,29\}, \ \{2,10,18,26\}, \ \{3,7,11,15,19,23,27,31\}, \\ 
			&\{4,20\}, \ \{6,14,22,30\}, \ \{8\}, \ \{12,28\} \ \{16\}, \ \{24\}.
		\end{align*}
		They are all equal-difference cosets.
		\item[(2)] Let $q = 3$ and $n=32$. All $q$-cyclotomic cosets modulo $n$ are 
		\begin{align*}
			&\{0\}, \ \{1,3,9,11,17,19,25,27\}, \ \{2,6,18,22\}, \ \{4,12\}, \\ &\{5,7,13,15,21,23,29,31\}, \ \{8,24\}, \ \{10,14,26,30\}, \ \{16\}, \ \{20,28\}.
		\end{align*}
		Among them the cosets $\{0\}$, $\{8,24\}$ and $\{16\}$ are of equal difference, while the rest are not.
	\end{description}
\end{examples}
	
\begin{lemma}\label{lem 4}
	Let $c_{n/q}(\gamma)$ be a $q$-cyclotomic coset modulo $n$ with $| c_{n/q}(\gamma)| = \tau$. Then $c_{n/q}(\gamma)$ is of equal difference if and only if
	\begin{description}
		\item[(\romannumeral1)] $\tau \mid n$; and
		\item[(\romannumeral2)] $\gamma q \equiv \gamma \pmod{\dfrac{n}{\tau}}$.
	\end{description}
\end{lemma}
	
\begin{proof}
	It is trivial to see that the conclusion holds for the case that $\tau=1$. In the following we assume that $\tau > 1$. If $c_{n/q}(\gamma)$ is an equal-difference cyclotomic coset, by definition it holds $\tau \mid n$. Since $\gamma q \in c_{n/q}(\gamma)$, then there is an integer $1 \leq k \leq \tau-1$ such that
	$$\gamma q \equiv \gamma + k \cdot \dfrac{n}{\tau} \pmod{n},$$
	which implies that $\gamma q \equiv \gamma \pmod{\frac{n}{\tau}}$.
	
	Conversely, suppose that $c_{n/q}(\tau)$ satisfies (\romannumeral1) and (\romannumeral2). By induction we have, for any positive integer $j$,
	$$\dfrac{n}{\tau} \mid \gamma q^{j} - \gamma.$$
	The elements $\gamma q^{j}-\gamma$, $j=0,1,\cdots,\tau-1$, are pairwise distinct in $\mathbb{Z}/n\mathbb{Z}$. Note that there are in total $\tau$ elements in $\mathbb{Z}/n\mathbb{Z}$ which are multiples of $\frac{n}{\tau}$, so they are exactly  $\gamma q^{j}-\gamma$, $j=0,1,\cdots,\tau-1$. It follows that 
	$$c_{n/q}(\gamma) = \{\gamma, \gamma+\dfrac{n}{\tau},\cdots,\gamma+(\tau-1)\dfrac{n}{\tau}\}.$$
\end{proof}

For any $q$-cyclotomic coset $c_{n/q}(\gamma)$ modulo $n$, we set $\widetilde{\gamma} = \frac{\gamma}{\mathrm{gcd}(\gamma,n)}$ and $n_{\gamma} = \frac{n}{\mathrm{gcd}(\gamma,n)}$. Then the $q$-cyclotomic coset modulo $n_{\gamma}$ containing $\widetilde{\gamma}$ is given by 
$$c_{n_{\gamma}/q}(\widetilde{\gamma}) = \{\widetilde{\gamma},\widetilde{\gamma}q,\cdots,\widetilde{\gamma}q^{\tau-1}\}.$$
There is a natural bijection
$$c_{n/q}(\gamma) \rightarrow c_{n_{\gamma}/q}(\widetilde{\gamma}): \ \gamma q^{j} \mapsto \widetilde{\gamma}q^{j}.$$
Notice that the coset $c_{n_{\gamma}/q}(\widetilde{\gamma})$ and the above bijection is independent of the choice of the representative $\gamma$. We call $c_{n_{\gamma}/q}(\widetilde{\gamma})$ the primitive form of $c_{n/q}(\gamma)$. In particular, if $\mathrm{gcd}(\gamma,n)=1$, then $c_{n/q}(\gamma)$ and $c_{n_{\gamma}/q}(\widetilde{\gamma})$ coincide. In this case $c_{n/q}(\gamma)$ is called a primitive $q$-cyclotomic coset modulo $n$.

\begin{lemma}\label{lem 1}
	Let $c_{n/q}(\gamma)$ be a $q$-cyclotomic coset modulo $n$. Then $c_{n/q}(\gamma)$ is an equal-difference cyclotomic coset if and only if its primitive form $c_{n_{\gamma}/q}(\widetilde{\gamma})$ is.
\end{lemma}

\begin{proof}
	Denote by $m= \mathrm{gcd}(\gamma,n)$. Writing the elements of $c_{n_{\gamma}/q}(\widetilde{\gamma})$ as nonnegative integers less than $n_{\gamma}$:
	\begin{equation}\label{eq 2}
		c_{n_{\gamma}/q}(\widetilde{\gamma}) = \{\widetilde{\gamma}_{0}, \widetilde{\gamma}_{0}+d_{1},\cdots,\widetilde{\gamma}_{0}+d_{1}+\cdots+d_{\tau-1}\},
	\end{equation}
	where $\widetilde{\gamma}_{0}$ is the leader of $c_{n_{\gamma}/q}(\widetilde{\gamma})$, and $d_{1},\cdots,d_{\tau-1} \geq 1$, then $c_{n/q}(\gamma)$ can be expressed as 
	\begin{equation}\label{eq 5}
		c_{n/q}(\gamma)= \{\widetilde{\gamma}_{0} m, \widetilde{\gamma}_{0} m+d_{1} m,\cdots,\widetilde{\gamma}_{0} m+(d_{1}+\cdots+d_{\tau-1}) m\}.
	\end{equation}
	Notice that the elements on the RHS of \eqref{eq 5} are all nonnegative integers less than $n$, therefore they form an arithmetic sequence if and only if the elements in \eqref{eq 2} form an arithmetic sequence. Furthermore, if it is this case, say, $d_{1} = \cdots = d_{\tau-1} = d$, then $\widetilde{\gamma}_{0} \equiv \widetilde{\gamma}_{0} + \tau d \pmod{n_{\gamma}}$ is equivalent to 
	$$\widetilde{\gamma}_{0}m \equiv \widetilde{\gamma}_{0}m + \tau dm \pmod{n}.$$
\end{proof}

Now we give a more applicable criterion for a cyclotomic coset to be of equal difference, which does not involve the size of the coset.

\begin{theorem}\label{thm 1}
	Let the notations be defined as above. A cyclotomic coset $c_{n/q}(\gamma)$ is of equal difference if and only if the following two conditions are satisfied:
	\begin{description}
		\item[(\romannumeral1)] $\mathrm{rad}(n_{\gamma}) \mid q-1$;
		\item[(\romannumeral2)] $q \equiv 1 \pmod{4}$ if $8 \mid n_{\gamma}$.
	\end{description}
\end{theorem}

\begin{proof}
	Applying Lemma \ref{lem 1}, without losing generality, we may assume that $\gamma$ is coprime to $n$ so that $n_{\gamma} = n$. First assume that the coset $c_{n/q}(\gamma)$, given by
	$$c_{n/q}(\gamma) = \{\gamma,\gamma q,\cdots,\gamma q^{\tau-1}\},$$
	satisfies (\romannumeral1) and (\romannumeral2). We treat the following cases separately.
	
	\textbf{Case $1$}: Let $n = p_{1}^{e_{1}}\cdots p_{s}^{e_{s}}$ be an odd integer, where $p_{1},\cdots,p_{s}$ are distinct odd primes coprime to $q$ and $e_{1},\cdots,e_{s}$ are positive integers. Since $\mathrm{rad}(n) = p_{1}\cdots p_{s} \mid q-1$, then
	$$d_{i} = v_{p_{i}}(q-1) > 0, \ i=1,\cdots,s.$$
	By the lift-the-exponent lemma we have
	$$\tau = p_{1}^{\mathrm{max}\{0,e_{1}-d_{1}\}}\cdots p_{s}^{\mathrm{max}\{0,e_{s}-d_{s}\}}.$$
	Clearly $\tau$ divides $n$, and 
	$$\dfrac{n}{\tau} = p_{1}^{\mathrm{min}\{e_{1},d_{1}\}}\cdots p_{s}^{\mathrm{min}\{e_{s},d_{s}\}} \mid q-1.$$
	It follows from Lemma \ref{lem 4} that $c_{n/q}(\gamma)$ is an equal-difference coset.
	
	\textbf{Case $2$}: Let $n = 2^{e_{0}}p_{1}^{e_{1}}\cdots p_{s}^{e_{s}}$ be an even integer, where $p_{1},\cdots,p_{s}$ are distinct odd primes and $e_{0},e_{1},\cdots,e_{s}$ are positive integers. Let $q$ be an odd prime power coprime to $n$, which satisfies that $q \equiv 1 \pmod{4}$. Denote by $d_{0} = v_{2}(q-1)$ and $d_{i} = v_{p_{i}}(q-1)$ for $i=1,\cdots,s$. Then $d_{0} \geq 2$ and $d_{1},\cdots,d_{s} \geq 1$. By the lift-the-exponent lemmas we have 
	$$\tau = 2^{\mathrm{max}\{0,e_{0}-d_{0}\}}p_{1}^{\mathrm{max}\{0,e_{1}-d_{1}\}}\cdots p_{s}^{\mathrm{max}\{0,e_{s}-d_{s}\}}.$$
	Clearly $\tau$ divides $n$, and 
	$$\dfrac{n}{\tau} = 2^{\mathrm{min}\{e_{0},d_{0}\}}p_{1}^{\mathrm{min}\{e_{1},d_{1}\}}\cdots p_{s}^{\mathrm{min}\{e_{s},d_{s}\}} \mid q-1.$$
	Hence in this case the coset $c_{n/q}(\gamma)$ is of equal difference.
	
	\textbf{Case $3$}: Let $n = 2^{e_{0}}p_{1}^{e_{1}}\cdots p_{s}^{e_{s}}$ be an even integer, where $p_{1},\cdots,p_{s}$ are pairwise distinct odd primes, $e_{0}$ is either $1$ or $2$, and $e_{1},\cdots,e_{s}$ are positive integers. Let $q$ be a prime power, which is coprime to $n$, such that $q \equiv 3 \pmod{4}$. Denote by $d_{i} = v_{p_{i}}(q-1)$ for $i = 1,\cdots,s$. Remembering that $v_{2}(q-1)=1$, then the lift-the-exponent lemmas implies that 
	\begin{equation*}
		\tau = \left\{
		\begin{array}{lcl}
			p_{1}^{\mathrm{max}\{0,e_{1}-d_{1}\}}\cdots p_{s}^{\mathrm{max}\{0,e_{s}-d_{s}\}}, \ \mathrm{if} \ e_{0}=1;\\
			2p_{1}^{\mathrm{max}\{0,e_{1}-d_{1}\}}\cdots p_{s}^{\mathrm{max}\{0,e_{s}-d_{s}\}}, \ \mathrm{if} \ e_{0}=2.
		\end{array} \right.
	\end{equation*}
    Thus one obtains that $\tau$ divides $n$, and for either $e_{0}=1$ or $e_{0}=2$,
    $$\dfrac{n}{\tau} = 2p_{1}^{\mathrm{min}\{e_{1},d_{1}\}}\cdots p_{s}^{\mathrm{min}\{e_{s},d_{s}\}} \mid q-1.$$
    Hence $c_{n/q}(\gamma)$ is an equal-difference cyclotomic coset.
    
    Conversely, suppose that $c_{n/q}(\gamma)$ is an equal-difference cyclotomic coset. If $\mathrm{rad}(n) \nmid q-1$, there exists a prime $\ell$ which divides $n$ but not $q-1$. In particular, $\ell$ and $\frac{n}{\tau}$ are coprime. Then there is an integer $k$ such that $k\cdot\frac{n}{\tau} \equiv -1 \pmod{\ell}$, or equivalently, $\ell \mid 1 + k\cdot \frac{n}{\tau}$. Since $\gamma + \gamma k \cdot \frac{n}{\tau}$ lies in $c_{n/q}(\gamma)$,
    $$\gamma + \gamma k \cdot \dfrac{n}{\tau} \equiv \gamma q^{j} \pmod{n}$$
    for some $0 \leq j \leq \tau-1$, which amounts to, by the assumption that $\mathrm{gcd}(\gamma,n)=1$, 
    $$1 + k \cdot \dfrac{n}{\tau} \equiv q^{j} \pmod{n}.$$
    As $\ell$ divides both $1+k\cdot\frac{n}{\tau}$ and $n$, $\ell$ divides $q^{j}$ and thus also $q$. It contradicts that $q^{\tau} \equiv 1 \pmod{n}$.
    
    Finally, suppose that $8 \mid n$ and $q \equiv 3 \pmod{4}$. Write $n = 2^{e_{0}}p_{1}^{e_{1}}\cdots p_{s}^{e_{s}}$, where $p_{1},\cdots,p_{s}$ are distinct odd primes, $e_{0} \geq 3$, and $e_{1},\cdots,e_{s}$ are positive integers. As being proved in the last paragraph, it holds that
    $$d_{i} = v_{p_{i}}(q-1) > 0, \ i=1,\cdots,s.$$
    Since $e_{0} \geq 3$, the size $\tau$ of $c_{n/q}(\gamma)$ must be even. Furthermore, by the lift-the-exponent lemma, $v_{2}(\tau)$ is the smallest positive even integer satisfying that
    $$ v_{2}(q+1) + v_{2}(\tau) = v_{2}(q^{\tau}-1) \geq e_{0}.$$
    Thus we have 
    \begin{equation*}
    	v_{2}(\tau) = \left\{
    	\begin{array}{lcl}
    		1, \quad \mathrm{if} \ e_{0}\leq v_{2}(q+1)+1;\\
    		e_{0}-v_{2}(q+1), \quad \mathrm{if} \ e_{0} > v_{2}(q+1)+1.
    	\end{array} \right.
    \end{equation*}
    which indicates that 
    $$v_{2}(\dfrac{n}{\tau}) = \mathrm{min}(e_{0}-1,v_{2}(q+1)) > 1 = v_{2}(q-1).$$
    As $\mathrm{gcd}(\gamma,n)=1$, one obtains $\frac{n}{\tau} \nmid \gamma(q-1)$. This is a contradiction. Here we complete the proof.
\end{proof}

Further, it also can be determined when all the $q$-cyclotomic cosets modulo $n$ are equal-difference cosets.

\begin{corollary}\label{coro 1}
	Let $q$ be a prime power and $n$ be a positive integer coprime to $q$. Then the $q$-cyclotomic cosets modulo $n$ are all equal-difference cosets if and only if the following two conditions hold:
	\begin{description}
		\item[(\romannumeral1)] $\mathrm{rad}(n) \mid q-1$;
		\item[(\romannumeral2)] $q \equiv 1 \pmod{4}$ if $8 \mid n$.
	\end{description}
\end{corollary}

\begin{proof}
    Assume that $q$ and $n$ meet condition (\romannumeral1) and (\romannumeral2). Let $c_{n/q}(\gamma)$ be any $q$-cyclotomic coset modulo $n$, and let $n_{\gamma} = \frac{n}{\mathrm{gcd}(\gamma,n)}$. Certainly $n_{\gamma}$ is a divisor of $n$, therefore it holds that $\mathrm{rad}(n_{\gamma}) \mid q-1$, and $q \equiv 1 \pmod{4}$ if $8 \mid n_{\gamma}$. By Theorem \ref{thm 1} the coset $c_{n/q}(\gamma)$ is of equal difference.
    
    Conversely, suppose that all the $q$-cyclotomic coset modulo $n$ are equal-difference cosets. Choose a primitive $q$-cyclotomic coset $c_{n/q}(\gamma)$ modulo $n$, then we have
    $$n_{\gamma} = \dfrac{n}{\mathrm{gcd}(\gamma,n)} = n.$$
    Now the conclusion again follows from Theorem \ref{thm 1}.
\end{proof}

\begin{remark}
	There is a tedious but more straightforward proof of Theorem \ref{thm 1} and Corollary \ref{coro 1}, which follows from 
the results on representatives and sizes of cyclotomic cosets given in \cite{Zhu} (Theorem 3.1., Proposition 3.1., Theorem 3.4. 
and Corollary 3.3. in \cite{Zhu}). In addition, this proof offers an explanation for the phenomenon that given $\mathrm{rad}(n) 
\mid q-1$, non-equal-difference $q$-cyclotomic coset modulo $n$ appear when and only when $q \equiv 3 \pmod{4}$ 
and $v_{2}(n) \geq 3$. In fact, there exist sequences in the $2$-adic $q$-cyclotomic system with non-equal-difference 
components if and only if $q \equiv 3 \pmod{4}$. And for any such sequence, the minimal degree where the component is 
not equal difference is exactly the stable degree plus $1$. Now the conclusion follows from that the stable degree of 
any sequence is not less than $2$.
\end{remark}

From the proof of Corollary \ref{coro 1} we also deduce the following consequence.

\begin{corollary}
	Let $q$ be a prime power and $n$ be a positive integer coprime to $q$. Then all the $q$-cyclotomic cosets modulo $n$ are equal-difference cosets if and only if a primitive one is.
\end{corollary}

\section{The multiple equal-difference structure of cyclotomic cosets}\label{sec 3}
In this section, we consider the general case. In fact, although a cyclotomic coset $c_{n/q}(\gamma)$ is not necessarily of equal difference, it turns out that $c_{n/q}(\gamma)$ can always be expressed as a disjoint union of equal-difference subsets. A partition of $c_{n/q}(\gamma)$ into disjoint equal-difference subsets is called an equal-difference decomposition of $c_{n/q}(\gamma)$. If, moreover, the partition is in the form
$$c_{n/q}(\gamma) = \bigsqcup_{j=0}^{\mathrm{gcd}(t,\tau)-1}c_{n/q^{t}}(\gamma q^{j}),$$
where $t \in \mathbb{N}^{+}$, then it is called a multiple equal-difference representation of $c_{n/q}(\gamma)$. The class of multiple equal-difference representations is particularly of interest to us. In this section, we determine all the equal-difference decompositions of any cyclotomic coset, and give an explicit characterization to the class of multiple equal-difference representations.

To make the statements precise, we first need some preparations. Let $t$ be a positive integer. For any $\gamma \in \mathbb{Z}/n\mathbb{Z}$, the $q^{t}$-cyclotomic coset $c_{n/q^{t}}(\gamma)$ is automatically a subset of the $q$-cyclotomic coset $c_{n/q}(\gamma)$. Denote by $t^{\prime}= \mathrm{gcd}(t,\tau)$, where $\tau = |c_{n/q}(\gamma)|$. Then there are integers $u$ and $v$ such that $ut + v\tau = t^{\prime}$, and thus
$$\gamma q^{t^{\prime}} \equiv \gamma q^{ut} \pmod{n},$$
which implies that $c_{n/q^{t^{\prime}}}(\gamma) \subseteq c_{n/q^{t}}(\gamma)$. The inverse conclusion is trivial, hence $c_{n/q^{t}}(\gamma) = c_{n/q^{t^{\prime}}}(\gamma)$. It follows immediately that the $q$-cyclotomic coset $c_{n/q}(\gamma)$ can be written as
\begin{equation}\label{eq 7}
	c_{n/q}(\gamma) = \bigsqcup_{j=0}^{t^{\prime}-1}c_{n/q^{t^{\prime}}}(\gamma q^{j}) = \bigsqcup_{j=0}^{t^{\prime}-1}c_{n/q^{t}}(\gamma q^{j}).
\end{equation}
The identity \eqref{eq 7} is called the $q^{t}$-cyclotomic decomposition of $c_{n/q}(\gamma)$.

The definition of equal-difference cyclotomic coset can be generalized naturally to any subset of a coset. Let $E$ be a subset of $c_{n/q}(\gamma)$, with $| E|=\tau_{E}$. Then $E$ is called an equal-difference subset of $c_{n/q}(\gamma)$, if $\tau_{E} \mid n$ and $E$ has the form
$$\{\gamma,\gamma+\dfrac{n}{\tau_{E}},\cdots,\gamma+(\tau_{E}-1)\dfrac{n}{\tau_{E}}\}.$$
The quotient $\frac{n}{\tau_{E}}$ is called the common difference of $E$. 

We define an order on the set of equal-difference decompositions of a given coset. Let
$$c_{n/q}(\gamma) = \bigsqcup_{i \in I}E_{i} = \bigsqcup_{j \in J}E_{j}^{\prime}$$
be two equal-difference decompositions of $c_{n/q}(\gamma)$. We say that $\bigsqcup\limits_{i \in I}E_{i}$ is coarser than $\bigsqcup\limits_{j \in J}E_{j}^{\prime}$ and denote by $\bigsqcup\limits_{i \in I}E_{i} \geq \bigsqcup\limits_{j \in J}E_{j}^{\prime}$, or equivalently, $\bigsqcup\limits_{j \in J}E_{j}^{\prime}$ is finer than $\bigsqcup\limits_{i \in I}E_{i}$ and denote by $\bigsqcup\limits_{j \in J}E_{j}^{\prime} \leq \bigsqcup\limits_{i \in I}E_{i}$, if the index set $J$ can be partitioned as $J = \bigsqcup\limits_{i \in I}J_{i}$ and 
$$E_{i} = \bigsqcup_{j \in J_{i}}E_{j}^{\prime}, \ \forall i \in I.$$

Now we give a closer description of equal-difference decompositions of a cyclotomic coset. Let $c_{n/q}(\gamma)$ be a $q$-cyclotomic coset modulo $n$ with $|c_{n/q}(\gamma)| = \tau$. The first example of an equal-difference decomposition of $c_{n/q}(\gamma)$ is trivial to see, as any one-element subset is of equal difference so that $c_{n/q}(\gamma)$ can be expressed as
\begin{equation}\label{eq 3}
	c_{n/q}(\gamma) = \bigsqcup_{j=0}^{\tau-1} \{\gamma q^{j}\}.
\end{equation}
Clearly the decomposition \eqref{eq 3} is the unique finest equal-difference decomposition of $c_{n/q}(\gamma)$. 

\begin{lemma}\label{lem 5}
	Let $c_{n/q}(\gamma)$ be a $q$-cyclotomic coset modulo $n$, and $E$ be an equal-difference subset of $c_{n/q}(\gamma)$ with $|E|=\tau_{E}$. Then $E$ is a $q^{t}$-cyclotomic coset modulo $n$, where $t$ is the smallest positive integer such that 
	$$\gamma q^{t} \equiv \gamma \pmod{\dfrac{n}{\tau_{E}}}.$$
\end{lemma}

\begin{proof}
	Without losing generality, we may suppose that $\gamma$ is contained in $E$. Then $E$ can be written as 
	$$\{\gamma,\gamma+\dfrac{n}{\tau_{E}},\cdots,\gamma+(\tau_{E}-1)\dfrac{n}{\tau_{E}}\}.$$
	Since $E \subseteq c_{n/q}(\gamma)$, for any $1 \leq k \leq \tau_{E}-1$, there is an integer $1 \leq j_{k} \leq \tau-1$ such that 
	$$\gamma + k\cdot \dfrac{n}{\tau_{E}} \equiv \gamma q^{j_{k}} \pmod{n},$$
	which implies that $\gamma q^{j_{k}} \equiv \gamma \pmod{\frac{n}{\tau_{E}}}$. By the definition of $t$ we have $t \mid j_{k}$ and thus $\gamma q^{j_{k}} \in c_{n/q^{t}}(\gamma)$. Consequently, $E$ is contained in $c_{n/q^{t}}(\gamma)$.
	 
	Conversely, as $\frac{n}{\tau_{E}} \mid \gamma q^{t} -\gamma$, by induction one obtains that 
	$$\dfrac{n}{\tau_{E}} \mid \gamma q^{t\cdot j}-\gamma, \ \forall j \in \mathbb{N}.$$
	It follows immediately that $c_{n/q^{t}}(\gamma) \subseteq E$. In conclusion, we have $E = c_{n/q^{t}}(\gamma)$.
\end{proof}

On the other hand, according to Theorem \ref{thm 1}, we obtain the opposite direction of Lemma \ref{lem 5}.

\begin{lemma}\label{lem 6}
	Let $c_{n/q}(\gamma)$ be a $q$-cyclotomic coset modulo $n$, with $n_{\gamma} = \frac{n}{\mathrm{gcd}(\gamma,n)}$. Then any $q^{t}$-cyclotomic coset modulo $n$ contained in $c_{n/q}(\gamma)$ is of equal difference if and only if
	\begin{description}
		\item[(\romannumeral1)] $\mathrm{rad}(n_{\gamma}) \mid q^{t}-1$;
		\item[(\romannumeral2)] $q^{t} \equiv 1 \pmod{4}$ if $8 \mid n_{\gamma}$.
	\end{description}
	In this case, the $q^{t}$-cyclotomic decomposition
	\begin{equation}\label{eq 6}
		c_{n/q}(\gamma) = \bigsqcup_{j=0}^{t^{\prime}-1}c_{n/q^{t}}(\gamma q^{j}),
	\end{equation}
	where $t^{\prime} = \mathrm{gcd}(t,\tau)$, is an equal-difference decomposition of $c_{n/q}(\gamma)$.
\end{lemma}

\begin{proof}
	Notice that for any element $\gamma q^{j} \in c_{n/q}(\gamma)$, as $q$ is coprime to $n$, $\mathrm{gcd}(\gamma q^{j},n) = \mathrm{gcd}(\gamma,n)$, and thus
	$$n_{\gamma q^{j}} = \dfrac{n}{\mathrm{gcd}(\gamma q^{j},n)} = \dfrac{n}{\mathrm{gcd}(\gamma,n)} = n_{\gamma}.$$
	Now Theorem \ref{thm 1} indicates the first assertion. The second assertion is a direct consequence.
\end{proof}

\begin{remark}\label{re 1}
	If a positive integer $t$ satisfies the conditions of Lemma \ref{lem 6}, then $t^{\prime} = \mathrm{gcd}(t,\tau)$ is the smallest positive integer such that 
	$$\gamma q^{j}\cdot q^{t^{\prime}} \equiv \gamma q^{j} \pmod{\dfrac{n}{|c_{n/q^{t}}(\gamma q^{j})|}}, \ \forall 0 \leq j \leq t^{\prime}-1.$$
	Applying Lemma \ref{lem 5} recovers the identity
	$$c_{n/q}(\gamma) = \bigsqcup_{j=0}^{t^{\prime}-1}c_{n/q^{t}}(\gamma q^{j}) = \bigsqcup_{j=0}^{t^{\prime}-1}c_{n/q^{t^{\prime}}}(\gamma q^{j}).$$
\end{remark}

\begin{proposition}\label{prop 1}
	Let
	\begin{equation*}
		\omega_{\gamma} = \left\{
		\begin{array}{lcl}
			2\mathrm{ord}_{\mathrm{rad}(n_{\gamma})}(q), \quad \mathrm{if} \ q^{\mathrm{ord}_{\mathrm{rad}(n_{\gamma})}(q)} \equiv 3 \pmod{4} \ \mathrm{and} \ 8 \mid n_{\gamma};\\
			\mathrm{ord}_{\mathrm{rad}(n_{\gamma})}(q), \quad \mathrm{otherwise}.
		\end{array} \right.
	\end{equation*}
	Then 
	\begin{equation}\label{eq 4}
		c_{n/q}(\gamma) = \bigsqcup_{j=0}^{\omega_{\gamma}-1}c_{n/q^{\omega_{\gamma}}}(\gamma q^{j})
	\end{equation}
	is the unique coarsest equal-difference decomposition of $c_{n/q}(\gamma)$.
\end{proposition}

\begin{proof}
	It is clear form the construction of $\omega_{\gamma}$ that $\mathrm{rad}(n_{\gamma}) \mid q^{\omega_{\gamma}}-1$ and $q^{\omega_{\gamma}} \equiv 1 \pmod{4}$ if $8 \mid n_{\gamma}$. As $\omega_{\gamma}$ divides $\tau$, by Lemma \ref{lem 6} the disjoint-union
	$$c_{n/q}(\gamma) = \bigsqcup_{j=0}^{\omega_{\gamma}-1}c_{n/q^{\omega_{\gamma}}}(\gamma q^{j})$$
	is an equal-difference decomposition of $c_{n/q}(\gamma)$.
	
	Now it remains to show that \eqref{eq 4} is the unique coarsest decomposition. Let
	$$c_{n/q}(\gamma) = \bigsqcup_{i \in I}E_{i}$$
	be any equal-difference decomposition of $c_{n/q}(\gamma)$. By Lemma \ref{lem 5} each $E_{i}$ is a $q^{t_{i}}$-cyclotomic coset modulo $n$, say, $E_{i} = c_{n/q^{t_{i}}}(\gamma_{i})$. And Lemma \ref{lem 6} implies that for any $i \in I$,
	\begin{description}
		\item[(\romannumeral1)] $\mathrm{rad}(n_{\gamma_{i}}) \mid q^{t_{i}}-1$;
		\item[(\romannumeral2)] $q^{t_{i}} \equiv 1 \pmod{4}$ if $8 \mid n_{\gamma_{i}}$.
	\end{description}
	Since all the $\gamma_{i}$'s lie in $c_{n/q}(\gamma)$, $n_{\gamma_{i}} = n_{\gamma}$ and therefore the $t_{i}$'s are multiples of $\omega_{\gamma}$. That is, $c_{n/q^{t_{i}}}(\gamma_{i}) \subseteq c_{n/q^{\omega_{\gamma}}}(\gamma_{i})$. And as $c_{n/q^{\omega_{\gamma}}}(\gamma q^{j})$, $j = 0,1,\cdots,\omega_{\gamma}-1$, partition the whole coset $c_{n/q}(\gamma)$, then $c_{n/q^{\omega_{\gamma}}}(\gamma_{i})$ coincides with some $c_{n/q^{\omega_{\gamma}}}(\gamma q^{j_{i}})$, $0 \leq j_{i} \leq \omega_{\gamma}-1$. Define the subset $I_{j}$ of $I$ by
	$$I_{j} = \{i \in I | \ \gamma_{i}\in c_{n/q^{\omega_{\gamma}}}(\gamma q^{j})\}.$$
	It is obvious that $I = \bigsqcup\limits_{j=0}^{\omega-1}I_{j}$. And since $c_{n/q}(\gamma) = \bigsqcup\limits_{i \in I}E_{i}$ is a disjoint union, one obtains that for any $0 \leq j \leq \omega_{\gamma}-1$,
	$$c_{n/q^{\omega_{\gamma}}}(\gamma q^{j}) = \bigsqcup_{j \in J_{i}}E_{i}.$$
	Here we complete the proof.
\end{proof}

\begin{corollary}
	Let $E$ be an equal-difference subset of $c_{n/q}(\gamma)$. Then $E$ must be contained in one of the subset $c_{n/q^{\omega}}(\gamma q^{j})$, $j=0,1,\cdots,\omega-1$. Equivalently, the subsets $c_{n/q^{\omega}}(\gamma q^{j})$ give rise to the longest arithmetic sequences in $c_{n/q}(\gamma)$.
\end{corollary}

\begin{proof}
	It is a direct consequence of the proof of Proposition \ref{prop 1}.
\end{proof}

The main theorem of this section can be stated as follow.

\begin{theorem}\label{thm 2}
	Let the notations be defined as above. Each equal-difference decomposition of a cyclotomic coset $c_{n/q}(\gamma)$ has the form
	$$c_{n/q}(\gamma) = \bigsqcup_{j=0}^{\omega-1}(\bigsqcup_{i \in I_{j}}c_{n/q^{t_{j,i}}}(\gamma_{j,i})),$$
	where $t_{j,i}= \frac{\tau}{|c_{n/q^{t_{j,i}}}(\gamma_{j,i})|}$ is a multiple of $\omega_{\gamma}$, and $\gamma_{j,i} \in c_{n/q^{\omega_{\gamma}}}(\gamma q^{j})$. Moreover, if it is required that each component of the decomposition has the same size, then it is in the form
	$$c_{n/q}(\gamma)= \bigsqcup_{j=0}^{t-1}c_{n/q^{t}}(\gamma q^{j}),$$
	where $t = \frac{\tau}{|c_{n/q^{t}}(\gamma q^{j})|}$ is a multiple of $\omega_{\gamma}$.
\end{theorem}

\begin{proof}
	The first assertion is obtained directly from the proof of Proposition \ref{prop 1} and Remark \ref{re 1}. Now we prove the second assertion. Let
	$$c_{n/q}(\gamma) = \bigsqcup_{j=0}^{\omega-1}(\bigsqcup_{i \in I_{j}}c_{n/q^{t_{j,i}}}(\gamma_{j,i})),$$ 
	be an equal-difference decomposition of $c_{n/q}(\gamma)$, where $t_{j,i}= \frac{\tau}{|c_{n/q^{t_{j,i}}}(\gamma_{j,i})|}$. If all the component $c_{n/q^{t_{j,i}}}(\gamma_{j,i}))$ have the same size, say $\mu$, then all the $t_{j,i}$'s are equaling to $t = \frac{\tau}{\mu}$. Since $\omega_{\gamma} \mid t$, and 
	$$c_{n/q^{\omega_{\gamma}}}(\gamma q^{j}) = \bigsqcup_{i \in I_{j}}c_{n/q^{t_{j,i}}}(\gamma_{j,i})),$$
	it is straightforward to check that $c_{n/q^t}(\gamma_{j,i})$, $i \in I_{j}$, ranges unrepeated over $c_{n/q^{t}}(\gamma q^{j+i\omega_{\gamma}})$, $i=0,1,\cdots,\frac{t}{\omega_{\gamma}}-1$. Thus we have
	$$c_{n/q}(\gamma) = \bigsqcup_{j=0}^{\omega_{\gamma}-1}(\bigsqcup_{i =0}^{\frac{t}{\omega_{\gamma}}-1}c_{n/q^{t}}(\gamma q^{j+i\omega_{\gamma}})) = \bigsqcup_{j=0}^{t-1}c_{n/q^{t}}(\gamma q^{j}).$$
\end{proof}

From Theorem \ref{thm 2} the equal-difference decompositions of a coset which are given by cyclotomic decompositions are exactly those whose components are of the same size. The class of such decompositions are mainly of interest to us.

\begin{definition}
	\begin{description}
		\item[(1)] Let $c_{n/q}(\gamma)$ be a $q$-cyclotomic coset modulo $n$. A multiple equal-difference representation of $c_{n/q}(\gamma)$ is an equal-difference decomposition of $c_{n/q}(\gamma)$ which is given by a $q^{t}$-cyclotomic decomposition for some $t \in \mathbb{N}^{+}$. We denote by $\mathcal{MER}(c_{n/q}(\gamma))$ the class of multiple equal-difference representations of $c_{n/q}(\gamma)$.
		\item[(2)] A multiple equal-difference representation of $\mathcal{C}_{n/q}$ is a tuple
		$$\left(c_{n/q}(\gamma) = \bigsqcup_{j=0}^{\mathrm{gcd}(t,\tau_{\gamma})-1}c_{n/q^{t}}(\gamma q^{j})\right)_{c_{n/q}(\gamma) \in \mathcal{C}_{n/q}},$$
		where $\tau_{\gamma} = |c_{n/q}(\gamma)|$, and $t$ is a positive integer such that for every $c_{n/q}(\gamma) \in \mathcal{C}_{n/q}$, the decomposition $c_{n/q}(\gamma) = \bigsqcup\limits_{j=0}^{\mathrm{gcd}(t,\tau_{\gamma})-1}c_{n/q^{t}}(\gamma q^{j})$ is a multiple equal-difference representation of $c_{n/q}(\gamma)$. The class of multiple equal-difference representations of $\mathcal{C}_{n/q}$ are denoted by $\mathcal{MER}_{n/q}$.
	\end{description}
\end{definition}

Notice that the order on the set of all equal-difference decompositions of $c_{n/q}(\gamma)$ restricts to an order on the space $\mathcal{MER}(c_{n/q}(\gamma))$. Further, the orders on all $\mathcal{MER}(c_{n/q}(\gamma))$'s induce a natural order on $\mathcal{MER}_{n/q}$:
$$\left(c_{n/q}(\gamma) = \bigsqcup_{j=0}^{\mathrm{gcd}(t_{1},\tau_{\gamma})-1}c_{n/q^{t_{1}}}(\gamma q^{j})\right)_{c_{n/q}(\gamma) \in \mathcal{C}_{n/q}} \leq \left(c_{n/q}(\gamma) = \bigsqcup_{j=0}^{\mathrm{gcd}(t_{2},\tau_{\gamma})-1}c_{n/q^{t_{2}}}(\gamma q^{j})\right)_{c_{n/q}(\gamma) \in \mathcal{C}_{n/q}}$$
if and only if $\bigsqcup\limits_{j=0}^{\mathrm{gcd}(t_{1},\tau_{\gamma})-1}c_{n/q^{t_{1}}}(\gamma q^{j}) \leq \bigsqcup\limits_{j=0}^{\mathrm{gcd}(t_{2},\tau_{\gamma})-1}c_{n/q^{t_{2}}}(\gamma q^{j})$ for every $c_{n/q}(\gamma) \in \mathcal{C}_{n/q}$.

Summarizing the obtained results, we give a characterization of the spaces $\mathcal{MER}(c_{n/q}(\gamma))$ and $\mathcal{MER}_{n/q}$.

\begin{theorem}\label{thm 3}
	\begin{description}
		\item[(1)] Let $c_{n/q}(\gamma)$ be a $q$-cyclotomic coset modulo $n$, with $|c_{n/q}(\gamma)| = \tau$. Let $\Sigma(\tau;\omega_{\gamma})$ be the set of the divisors of $\tau$ that is divided by $\omega_{\gamma}$. Then there is an one-to-one correspondence
		$$\varphi_{c_{n/q}(\gamma)}: \Sigma(\tau;\omega_{\gamma}) \rightarrow \mathcal{MER}(c_{n/q}(\gamma)): \ t\mapsto c_{n/q}(\gamma) = \bigsqcup_{j=0}^{t-1}c_{n/q^{t}}(\gamma q^{j}) .$$
		Moreover, for any $t_{1}, t_{2} \in \Sigma(\tau;\omega_{\gamma})$, $t_{1}\mid t_{2}$ if and only if $\varphi_{c_{n/q}(\gamma)}(t_{2}) \leq \varphi_{c_{n/q}(\gamma)}(t_{1})$.
		\item[(2)] Let
		\begin{equation*}
			\omega = \left\{
			\begin{array}{lcl}
				2\mathrm{ord}_{\mathrm{rad}(n)}(q), \quad \mathrm{if} \ q^{\mathrm{ord}_{\mathrm{rad}(n)}(q)} \equiv 3 \pmod{4} \ \mathrm{and} \ 8 \mid n;\\
				\mathrm{ord}_{\mathrm{rad}(n)}(q), \quad \mathrm{otherwise}.
			\end{array} \right.
		\end{equation*}
		Then there is an one-to-one correspondence 
		$$\varphi_{n/q}: \Sigma(\mathrm{ord}_{n}(q);\omega)\rightarrow \mathcal{MER}_{n/q}; \ t\mapsto \left(c_{n/q}(\gamma) = \bigsqcup_{j=0}^{\mathrm{gcd}(t,\tau_{\gamma})-1}c_{n/q^{t}}(\gamma q^{j})\right)_{c_{n/q}(\gamma) \in \mathcal{C}_{n/q}}.$$
		Moreover, for any $t_{1},t_{2} \in \Sigma(\mathrm{ord}_{n}(q);\omega)$, $t_{1} \mid t_{2}$ if and only if $\varphi_{n/q}(t_{2}) \leq \varphi_{n/q}(t_{1})$.
	\end{description}
\end{theorem}

\begin{remark}
	If one admits the natural order on the set $\Sigma(\tau,\omega_{\gamma})$ (resp. $\Sigma(\mathrm{ord}_{n}(q);\omega)$): for any $t_{1},t_{2} \in \Sigma(\tau,\omega_{\gamma})$ (resp. $t_{1},t_{2} \in \Sigma(\mathrm{ord}_{n}(q);\omega)$), $t_{1} \leq t_{2}$ if and only if $t_{1} \mid t_{2}$, then Theorem \ref{thm 3} can be rephrased simply as: The map $\varphi_{c_{n/q}(\gamma)}$ (resp. $\varphi_{n/q}$) is an anti-order-preserving one-to-one correspondence.
\end{remark}

\begin{proof}
	\begin{description}
		\item[(1)] By Theorem \ref{thm 2} and Remark \ref{re 1} the map $\varphi_{c_{n/q}(\gamma)}$ is a bijection. Let $t_{1},t_{2} \in \Sigma(\tau;\omega_{\gamma})$. If
		$$\varphi_{c_{n/q}(\gamma)}(t_{2}) = \bigsqcup_{j=0}^{t_{2}-1}c_{n/q^{t_{2}}}(\gamma q^{j}) \leq \bigsqcup_{j=0}^{t_{1}-1}c_{n/q^{t_{1}}}(\gamma q^{j}) = \varphi_{c_{n/q}(\gamma)}(t_{1}),$$
		then 
		$$|c_{n/q^{t_{2}}}(\gamma)| = \dfrac{\tau}{t_{2}} \mid \dfrac{\tau}{t_{1}} = |c_{n/q^{t_{1}}}(\gamma)|,$$
		which implies that $t_{1} \mid t_{2}$.
		
		Conversely, if $t_{1} \mid t_{2}$, then for $0 \leq j \leq t_{1}$ the coset $c_{n/q^{t_{1}}}(\gamma q^{j})$ has the $q^{t_{2}}$-cyclotomic decomposition
		$$c_{n/q^{t_{1}}}(\gamma q^{j}) = \bigsqcup_{i=0}^{\frac{t_{2}}{t_{1}}-1}c_{n/q^{t_{2}}}(\gamma q^{j+it_{1}}).$$
		Hence we have
		$$c_{n/q}(\gamma) = \bigsqcup_{j=0}^{t_{1}-1}c_{n/q^{t_{1}}}(\gamma q^{j}) = \bigsqcup_{j=0}^{t_{1}-1}(\bigsqcup_{i=0}^{\frac{t_{2}}{t_{1}}-1}c_{n/q^{t_{2}}}(\gamma q^{j+it_{1}})).$$
		\item[(2)] Given any $q$-cyclotomic coset $c_{n/q}(\eta)$ modulo $n$, there is a natural projection $\psi_{\eta}: \mathcal{MER}_{n/q}\rightarrow \mathcal{MER}(c_{n/q}(\eta))$ given by
		$$\left(c_{n/q}(\gamma) = \bigsqcup_{j=0}^{\mathrm{gcd}(t,\tau_{\gamma})-1}c_{n/q^{t}}(\gamma q^{j})\right)_{c_{n/q}(\gamma) \in \mathcal{C}_{n/q}} \mapsto c_{n/q}(\eta) = \bigsqcup_{j=0}^{\mathrm{gcd}(t,\tau_{\eta})-1}c_{n/q^{t}}(\eta q^{j}).$$
		On the other hand, sending $t \in \Sigma(\mathrm{ord}_{n}(q);\omega)$ to $\mathrm{gcd}(t,\tau_{\eta}) \in \Sigma(\tau_{\eta};\omega_{\eta})$ defines a surjection $\psi_{\eta}^{\prime}$. It is trivial to see that the following diagram commutes:
		\begin{equation}
			\begin{tikzcd}
				\Sigma(\mathrm{ord}_{n}(q);\omega) \arrow{r}{\varphi_{n/q}}  \arrow{d}{\psi_{\eta}^{\prime}}& \mathcal{MER}_{n/q} \arrow{d}{\psi_{\eta}}\\
				\Sigma(\tau_{\eta};\omega_{\eta}) \ar{r}{\varphi_{c_{n/q}(\eta)}} & \mathcal{MER}(c_{n/q}(\eta))
			\end{tikzcd}
		\end{equation}
		Now let $\eta$ range over all representatives of $q$-cyclotomic cosets modulo $n$, then Conclusion $(2)$ is obtained from $(1)$.
	\end{description}
\end{proof}

\section{Irreducible factorization in binomial form of $X^{n}-1$}\label{sec 4}
Let $q$ be a prime power, and $n$ be a positive integer coprime to $q$. In many computation problems, it is convenient if $X^{n}-1$ factors into irreducible binomials over $\mathbb{F}_{q}$. In this case, we say that the irreducible factorization of $X^{n}-1$ over $\mathbb{F}_{q}$ is in binomial form. The main goal of this section is to construct the equivalent correspondence, between the multiple equal-difference representations of $\mathcal{C}_{n/q}$ and the irreducible factorizations of $X^{n}-1$ in binomial form over extension fields of $\mathbb{F}_{q}$. Through the correspondence the results obtained in Section \ref{sec 2} and \ref{sec 3} can be translated into the problem of factorizing $X^{n}-1$.

Recall that we fix a family of primitive roots of unity
$$\{\zeta_{n}| \ \mathrm{gcd}(n,q)=1\}$$
in the algebraic closure $\overline{\mathbb{F}_{q}}$ over $\mathbb{F}_{q}$, which satisfies that for any integers $m$ and $n$, coprime to $q$, such that $m \mid n$, it holds that $\zeta_{n}^{\frac{n}{m}} = \zeta_{m}$. Let
$$c_{n/q}(\gamma) = \{\gamma,\gamma q,\cdots,\gamma q^{\tau-1}\}$$
be a $q$-cyclotomic coset modulo $n$. Then $c_{n/q}(\gamma)$ induces an irreducible factor of $X^{n}-1$:
$$M_{c_{n/q}(\gamma)}(X) = (X-\zeta_{n}^{\gamma})(X-\zeta_{n}^{\gamma q})\cdots (X-\zeta_{n}^{\gamma q^{\tau-1}}).$$
Moreover, the irreducible factorization of $X^{n}-1$ over $\mathbb{F}_{q}$ is given by
$$X^{n}-1 = \prod_{\gamma \in \mathcal{CR}_{n/q}} M_{c_{n/q}(\gamma)}(X),$$
where $\mathcal{CR}_{n/q}$ is any full set of representatives of $q$-cyclotomic cosets modulo $n$.

\begin{lemma}\label{lem 7}
	Let $n_{1}$ and $n_{2}$ be positive integers coprime to $q$. Let $\gamma_{1} \in \mathbb{Z}/n_{1}\mathbb{Z}$ and $\gamma_{2} \in \mathbb{Z}/n_{2}\mathbb{Z}$. The cosets $c_{n_{1}/q}(\gamma_{1})$ and $c_{n_{2}/q}(\gamma_{2})$ give rise to the same irreducible polynomial over $\mathbb{F}_{q}$ if and only if they have the same primitive form.
\end{lemma}

\begin{proof}
	First we prove the following claim: For any coset $c_{n/q}(\gamma)$ and any positive integer $d$ coprime to $q$, $M_{c_{dn/q}(d\gamma)}(X) = M_{c_{n/q}(\gamma)}(X)$. It is obvious that $M_{c_{dn/q}(d\gamma)}(X)$ and $M_{c_{n/q}(\gamma)}(X)$ have the same size, say $\tau$. Since $\zeta_{dn}^{d\gamma} = \zeta_{n}^{\gamma}$, then it holds that
	\begin{align*}
		M_{c_{dn/q}(d\gamma)}(X) &= (X-\zeta_{dn}^{d\gamma})(X-\zeta_{dn}^{d\gamma q})\cdots (X-\zeta_{dn}^{d\gamma q^{\tau-1}})\\ 
		&= (X-\zeta_{n}^{\gamma})(X-\zeta_{n}^{\gamma q})\cdots (X-\zeta_{n}^{\gamma q^{\tau-1}}) = M_{c_{n/q}(\gamma)}(X).
	\end{align*}
	
	Now if $c_{n_{1}/q}(\gamma_{1})$ and $c_{n_{2}/q}(\gamma_{2})$ have the same primitive form, by the above claim
	$$M_{c_{n_{1}/q}(\gamma_{1})}(X) = M_{c_{n_{1,\gamma_{1}}/q}(\widetilde{\gamma_{1}})}(X) = M_{c_{n_{2,\gamma_{2}}/q}(\widetilde{\gamma_{2}})}(X) = M_{c_{n_{2}/q}(\gamma_{2})}(X),$$
	where $n_{i,\gamma_{i}} = \frac{n_{i}}{\mathrm{gcd}(n_{i},\gamma_{i})}$ for $i=1,2$. Conversely, suppose that $M_{c_{n_{1}/q}(\gamma_{1})}(X) =M_{c_{n_{2}/q}(\gamma_{2})}(X)$. Let $n = \mathrm{lcm}(n_{1},n_{2})$ and $d_{i} = \frac{n}{n_{i}}$ for $i = 1,2$. Also by the above claim one obtains 
	$$M_{c_{d_{1}n_{1}/q}(d_{1}\gamma_{1})}(X) =M_{c_{d_{2}n_{2}/q}(d_{2}\gamma_{2})}(X).$$
	Note that $X^{n}-1$ does not have repeated factor, therefore $c_{n/q}(d_{1}\gamma_{1}) = c_{n/q}(d_{2}\gamma_{2})$. It follows immediately that they have the same primitive form.
\end{proof}

\begin{theorem}\label{thm 4}
	Let $c_{n/q}(\gamma)$ be a $q$-cyclotomic coset modulo $n$, and $M_{c_{n/q}(\gamma)}(X)$ be the irreducible polynomial induced by $c_{n/q}(\gamma)$. Then $M_{c_{n/q}(\gamma)}(X)$ is a binomial if and only if $c_{n/q}(\gamma)$ is of equal difference.
\end{theorem}

\begin{proof}
	Assume that $c_{n/q}(\gamma)$ is an equal-difference coset, with common difference $d$. As for the case that $n=d$, which is equivalent to that $c_{n/q}(\gamma)$ is an one-element-set, the conclusion is trivial, without losing generality we suppose that $n > d$. The coset $c_{n/q}(\gamma)$ can be written as
	$$c_{n/q}(\gamma) = \{\gamma,\gamma+d,\cdots,\gamma+(\tau-1)d\},$$
	where $\tau d =n$. Expand $M_{c_{n/q}(\gamma)}(X)$ as
	$$M_{c_{n/q}(\gamma)}(X) = (X-\zeta_{n}^{\gamma})(X-\zeta_{n}^{\gamma+d})\cdots(X-\zeta_{n}^{\gamma+(\tau-1)d}) = X^{\tau} + a_{1}X^{\tau-1} + \cdots + a_{\tau-1}X +a_{\tau}.$$
	We prove by induction that $a_{1} = a_{2} = a_{\tau-1} = 0$. First, it can be computed directly that
	$$a_{1} = -(\zeta_{n}^{\gamma} + \cdots + \zeta_{n}^{\gamma+(\tau-1)d}) = -\zeta_{n}^{\gamma}\cdot \dfrac{\zeta_{n}^{n}-1}{\zeta_{n}^{d}-1} = 0,$$
	as $\zeta_{n}^{d}-1 \neq 0$. Now suppose that $a_{1} = \cdots = a_{k} = 0$, where $1 \leq k < \tau-1$. Note that for any $1 \leq \ell < \tau$,
	$$a_{\ell} = (-1)^{\ell}\sum_{0 \leq j_{1} < \cdots < \cdots j_{\ell} \leq \tau-1}\zeta_{n}^{\gamma+j_{1}d}\cdots\zeta_{n}^{\gamma+j_{\ell}d} = (-1)^{\ell}\dfrac{1}{(\ell)!}\zeta_{n}^{\ell\gamma}\cdot\sideset{}{^{\prime}}\sum_{j_{1},\cdots,j_{\ell}}\zeta_{n}^{j_{1}d}\cdots\zeta_{n}^{j_{\ell}d},$$
	where the sum $\sideset{}{^{\prime}}\sum\limits_{j_{1},\cdots,j_{\ell}}$ takes all $\ell$-tuples $(j_{1},\cdots,j_{\ell}) \in \{0,1,\cdots,\tau-1\}^{\ell}$ such that $j_{1},\cdots,j_{\ell}$ are pairwise distinct. As $a_{k}= (-1)^{k}\dfrac{1}{k!}\zeta_{n}^{k\gamma}\cdot\sideset{}{^{\prime}}\sum\limits_{j_{1},\cdots,j_{k}}\zeta_{n}^{j_{1}d}\cdots\zeta_{n}^{j_{k}d} = 0$, therefore
	\begin{align*}
		\sideset{}{^{\prime}}\sum\limits_{j_{1},\cdots,j_{k+1}}\zeta_{n}^{j_{1}d}\cdots\zeta_{n}^{j_{k+1}d} &= \sum_{j_{1}=0}^{\tau-1}\zeta_{n}^{j_{1}d}(\sideset{}{^{\prime}}\sum\limits_{\substack{j_{2},\cdots,j_{k+1}\\ j_{i} \neq j_{1}}}\zeta_{n}^{j_{2}d}\cdots\zeta_{n}^{j_{k+1}d})\\
		&= \sum_{j_{1}=0}^{\tau-1}\zeta_{n}^{j_{1}d}(-k\cdot\sideset{}{^{\prime}}\sum\limits_{\substack{j_{2},\cdots,j_{k}\\ j_{i} \neq j_{1}}}\zeta_{n}^{j_{1}d}\zeta_{n}^{j_{2}d}\cdots\zeta_{n}^{j_{k}d})\\
		&= -k\cdot \sum_{j_{1}=0}^{\tau-1}\zeta_{n}^{2j_{1}d}(\sideset{}{^{\prime}}\sum\limits_{\substack{j_{2},\cdots,j_{k}\\ j_{i} \neq j_{1}}}\zeta_{n}^{j_{2}d}\cdots\zeta_{n}^{j_{k}d}),
	\end{align*}
	where the second equality holds because $\sideset{}{^{\prime}}\sum\limits_{j_{1},\cdots,j_{k}}\zeta_{n}^{j_{1}d}\cdots\zeta_{n}^{j_{k}d} = 0$ and each term $\zeta_{n}^{j_{1}d}\zeta_{n}^{j_{2}d}\cdots\zeta_{n}^{j_{k}d}$ for a fixed $j_{1}$ appears as a summand in $\sideset{}{^{\prime}}\sum\limits_{j_{1},\cdots,j_{k}}\zeta_{n}^{j_{1}d}\cdots\zeta_{n}^{j_{k}d}$ by $k$ times. Applying the same argument successively with $a_{k-1}= \cdots = a_{1}=0$ yields that 
	$$\sideset{}{^{\prime}}\sum\limits_{j_{1},\cdots,j_{k+1}}\zeta_{n}^{j_{1}d}\cdots\zeta_{n}^{j_{k+1}d} = (-1)^{k}k!\cdot \sum_{j_{1}=0}^{\tau-1}\zeta_{n}^{(k+1)j_{1}d},$$
	which indicates that 
	$$a_{k+1}= -\frac{1}{k+1}\zeta_{n}^{(k+1)\gamma}\cdot\sum\limits_{j=0}^{\tau-1}\zeta_{n}^{(k+1)jd}.$$
	Since $k+1 < \tau$, then $\zeta_{n}^{(k+1)d} \neq 1$ and 
	$$a_{k+1}= -\frac{1}{k+1}\zeta_{n}^{(k+1)\gamma}\cdot\dfrac{\zeta_{n}^{(k+1)\tau d}-1}{\zeta_{n}^{(k+1)d}-1} = 0.$$
	By induction it holds that $a_{1} = a_{2} = a_{\tau-1} = 0$.
	
	To prove the opposite direction, we assume that $c_{n/q}(\gamma) = \{\gamma,\gamma q,\cdots,\gamma q^{\tau-1}\}$ induces a binomial, that is,
	$$M_{c_{n/q}(\gamma)}(X) = (X-\zeta_{n}^{\gamma})(X-\zeta_{n}^{\gamma q})\cdots (X-\zeta_{n}^{\gamma q^{\tau-1}}) = X^{\tau} - \zeta_{n}^{\gamma\cdot\frac{q^{\tau}-1}{q-1}}.$$
	Since $n \mid \gamma q^{\tau} -\gamma$, the order $e = \mathrm{ord}(\zeta_{n}^{\gamma\cdot\frac{q^{\tau}-1}{q-1}})$ is a divisor of $q-1$. Note that
	$$X^{\tau}-\zeta_{n}^{\gamma\cdot\frac{q^{\tau}-1}{q-1}} \mid X^{\tau e}-1.$$
	Then $\frac{\tau e}{\tau} = e \mid q-1$ implies that the $q$-cyclotomic coset modulo $\tau e$ which induces $M_{c_{n/q}(\gamma)}(X)$ is of equal difference. Now the conclusion follows from Lemma \ref{lem 1} and \ref{lem 7}.
\end{proof}

Combining Theorem \ref{thm 1}, Corollary \ref{coro 1} and Theorem \ref{thm 4} yields the following criteria. In particular, the second part of Corollary \ref{coro 2} recovers the equivalent condition for $X^{n}-1$ factoring into irreducible binomials over $\mathbb{F}_{q}$ given in \cite{Martinez}.

\begin{corollary}\label{coro 2}
	Let the notations be given as above. Then we have
	\begin{description}
		\item[(1)] The polynomial $M_{c_{n/q}(\gamma)}(X)$ is a binomial if and only if
		\begin{description}
			\item[(\romannumeral1)] $\mathrm{rad}(n_{\gamma}) \mid q-1$;
			\item[(\romannumeral2)] $q \equiv 1 \pmod{4}$ if $8 \mid n_{\gamma}$.
		\end{description}
		\item[(2)] The polynomial $X^{n}-1$ factors into irreducible binomials over $\mathbb{F}_{q}$ if and only if 
		\begin{description}
			\item[(\romannumeral1)] $\mathrm{rad}(n) \mid q-1$;
			\item[(\romannumeral2)] $q \equiv 1 \pmod{4}$ if $8 \mid n$.
		\end{description}
	\end{description}
\end{corollary}

Now we turn to the general case. Let $c_{n/q}(\gamma)$ be a $q$-cyclotomic coset modulo $n$, and $\tau = |c_{n/q}(\gamma)|$. Let $t$ be a positive integer. Then $c_{n/q}(\gamma)$ has the $q^{t}$-cyclotomic decomposition
$$c_{n/q}(\gamma) = \bigsqcup_{j=0}^{t^{\prime}-1}c_{n/q^{t}}(\gamma q^{j}),$$
where $t^{\prime} = \mathrm{gcd}(t,\tau)$. Notice that the $q^{t}$-decomposition of $c_{n/q}(\gamma)$ give rise to the irreducible factorization
$$M_{c_{n/q}(\gamma)}(X) = \prod_{j=0}^{t^{\prime}-1}M_{c_{n/q^{t}}(\gamma q^{j})}(X)$$
of $M_{c_{n/q}(\gamma)}(X)$ over $\mathbb{F}_{q^{t}}$. Then by Theorem \ref{thm 4} $M_{c_{n/q}(\gamma)}(X)$ factors into irreducible binomials over $\mathbb{F}_{q^{t}}$ if and only if $c_{n/q^{t}}(\gamma q^{j})$, $j=0,1,\cdots,t^{\prime}-1$, are all equal-difference $q^{t}$-cyclotomic cosets modulo $n$, that is, the $q^{t}$-cyclotomic decomposition of $c_{n/q}(\gamma)$ is a multiple equal-difference representation. Hence we prove the following proposition.

\begin{proposition}\label{prop 2}
	Let $c_{n/q}(\gamma)$ be a $q$-cyclotomic coset modulo $n$. Then the irreducible factorization of $M_{c_{n/q}(\gamma)}(X)$ over $\mathbb{F}_{q^{t}}$ is in binomial form if and only if the $q^{t}$-cyclotomic decomposition of $c_{n/q}(\gamma)$ is a multiple equal-difference decomposition.
\end{proposition}

For any positive integer $t$, the irreducible factorization of $M_{c_{n/q}(\gamma)}(X)$ (resp. $X^{n}-1$) over $\mathbb{F}_{q^{t}}$ is the same as that over $\mathbb{F}_{q^{t^{\prime}}}$, where $t^{\prime} = \mathrm{gcd}(t,\tau)$ (resp. $t^{\prime} = \mathrm{gcd}(t,\mathrm{ord}_{n}(q))$). Therefore it does not lose generality that we restrict the attention to divisors of $\tau$ (resp. $\mathrm{ord}_{n}(q)$). Then Proposition \ref{prop 2} gives the following equivalent correspondences.

\begin{theorem}\label{thm 5}
	\begin{description}
		\item[(1)] There is a one-to-one correspondence from $\mathcal{MER}(c_{n/q}(\gamma))$ onto the set of extension fields of $\mathbb{F}_{q}$ contained in $\mathbb{F}_{q^{\tau}}$ where $M_{c_{n/q}(\gamma)}(X)$ has irreducible factorization in binomial:
		$$\chi_{c_{n/q}(\gamma)}: c_{n/q}(\gamma) = \bigsqcup_{j=0}^{t-1}c_{n/q^{t}}(\gamma q^{j}) \mapsto \mathbb{F}_{q^{t}},$$
		where $t$ is a divisor of $\tau$ that is divided by $\omega_{\gamma}$. Moreover, for any multiple equal-difference representations 
		$$c_{n/q}(\gamma) = \bigsqcup_{j=0}^{t_{1}-1}c_{n/q^{t_{1}}}(\gamma q^{j}) = \bigsqcup_{j=0}^{t_{2}-1}c_{n/q^{t_{2}}}(\gamma q^{j}),$$
		the representations satisfy $\bigsqcup\limits_{j=0}^{t_{2}-1}c_{n/q^{t_{2}}}(\gamma q^{j}) \leq \bigsqcup\limits_{j=0}^{t_{1}-1}c_{n/q^{t_{1}}}(\gamma q^{j})$ if and only if $\mathbb{F}_{q^{t_{1}}} \subseteq \mathbb{F}_{q^{t_{2}}}$.
		\item[(2)] There is a one-to-one correspondence from $\mathcal{MER}_{n/q}$ onto the set of extension fields of $\mathbb{F}_{q}$ contained in $\mathbb{F}_{q^{\mathrm{ord}_{n}(q)}}$ where $X^{n}-1$ has irreducible factorization in binomial:
		$$\chi_{n/q}: \left(c_{n/q}(\gamma) = \bigsqcup_{j=0}^{\mathrm{gcd}(t,\tau_{\gamma})-1}c_{n/q^{t}}(\gamma q^{j})\right)_{c_{n/q}(\gamma) \in \mathcal{C}_{n/q}} \mapsto \mathbb{F}_{q^{t}},$$
		where $t$ is a divisor of $\mathrm{ord}_{n}(q)$ that is divided by $\omega$. Moreover, for any multiple equal-difference representations 
		$$\left(c_{n/q}(\gamma) = \bigsqcup_{j=0}^{\mathrm{gcd}(t_{1},\tau_{\gamma})-1}c_{n/q^{t_{1}}}(\gamma q^{j})\right)_{c_{n/q}(\gamma) \in \mathcal{C}_{n/q}}, \ \left(c_{n/q}(\gamma) = \bigsqcup_{j=0}^{\mathrm{gcd}(t_{2},\tau_{\gamma})-1}c_{n/q^{t_{2}}}(\gamma q^{j})\right)_{c_{n/q}(\gamma) \in \mathcal{C}_{n/q}}$$
		in $\mathcal{MER}_{n/q}$, the first is coarser than the latter if and only if $\mathbb{F}_{q^{t_{1}}} \subseteq \mathbb{F}_{q^{t_{2}}}$.
	\end{description}
\end{theorem}

\begin{corollary}
	\begin{description}
		\item[(1)] For any positive integer $t$, the induced polynomial $M_{c_{n/q}(\gamma)}(X)$ factors into irreducible binomials over $\mathbb{F}_{q^{t}}$ if and only if $\omega_{\gamma}\mid t$.
		\item[(2)] For any positive integer $t$, $X^{n}-1$ factors into irreducible binomials over $\mathbb{F}_{q^{t}}$ if and only if $\omega\mid t$.
	\end{description}
\end{corollary}

\section{Leaders of cyclotomic cosets}\label{sec 5}
In this section, as an application of the multiple equal-difference representations of cyclotomic cosets, we present an 
algorithm to simplify the computation of the leaders of cyclotomic cosets. In particular, the leader of any equal-difference 
coset is determined. Here we introduce the following notation. Let $n$ be a positive integer. For any integer $m$, 
we denote by $m\pmod{n}$ the unique nonnegative integer $m^{\prime}$ less than $n$ such that $m^{\prime}\equiv m\pmod n$.

\begin{lemma}\label{lem 8}
	Let $c_{n/q}(\gamma)$ be a $q$-cyclotomic coset modulo $n$ with $|c_{n/q}(\gamma)|=\tau$. Assume that $c_{n/q}(\gamma)$ is of equal difference. Then the leader of $c_{n/q}(\gamma)$ is $\gamma(\mathrm{mod} \ \frac{n}{\tau})$.
\end{lemma}

\begin{proof}
	The conclusion is trivial in the case where $\tau=1$. In the following we assume that $\tau > 1$. Write the elements in $c_{n/q}(\gamma)$ as nonnegative integers less than $n$:
	$$c_{n/q}(\gamma) = \{\gamma_{0},\gamma_{0}+\dfrac{n}{\tau},\cdots,\gamma_{0}+(\tau-1)\dfrac{n}{\tau}\},$$
	where $\gamma_{0}$ is the leader of $c_{n/q}(\gamma)$. Since $0 < \gamma_{0}+(\tau-1)\dfrac{n}{\tau} < n$, $\gamma_{0}$ is the unique element in $c_{n/q}(\gamma)$ lying in the range $0 \leq \gamma_{0} < \frac{n}{\tau}$. As $\gamma_{0} \equiv \gamma \pmod{\frac{n}{\tau}}$, the conclusion holds.
\end{proof}

For the general case, let $c_{n/q}(\gamma)$ be a $q$-cyclotomic coset modulo $n$, with $n_{\gamma} = \frac{n}{\mathrm{gcd}(n,\gamma)}$. Let
\begin{equation*}
	\omega_{\gamma} = \left\{
	\begin{array}{lcl}
		2\mathrm{ord}_{\mathrm{rad}(n_{\gamma})}(q), \quad \mathrm{if} \ q^{\mathrm{ord}_{\mathrm{rad}(n_{\gamma})}(q)} \equiv 3 \pmod{4} \ \mathrm{and} \ 8 \mid n_{\gamma};\\
		\mathrm{ord}_{\mathrm{rad}(n_{\gamma})}(q), \quad \mathrm{otherwise}.
	\end{array} \right.
\end{equation*}
Then by Theorem \ref{thm 2}
$$c_{n/q}(\gamma) = \bigsqcup_{j=0}^{\omega_{\gamma}-1}c_{n/q^{\omega_{\gamma}}}(\gamma q^{j})$$
gives the coarsest equal-difference decomposition of $c_{n/q}(\gamma)$. For each component $c_{n/q^{\omega_{\gamma}}}(\gamma q^{j})$, Lemma \ref{lem 8} gives the leader as $\gamma q^{j}(\mathrm{mod} \ \frac{\omega_{\gamma}n}{\tau})$. Then the leader of the whole coset $c_{n/q}(\gamma)$ must be the smallest one among $\gamma q^{j}(\mathrm{mod} \ \frac{\omega_{\gamma}n}{\tau})$, $j=0,1,\cdots,\omega_{\gamma}-1$. Hence we prove the following theorem.

\begin{theorem}\label{thm 6}
	With the notations as above, the leader of the $q$-cyclotomic coset $c_{n/q}(\gamma)$ is 
	$$\mathrm{min}\{\gamma(\mathrm{mod} \ \frac{\omega_{\gamma}n}{\tau}), \gamma q(\mathrm{mod} \ \frac{\omega_{\gamma}n}{\tau}),\cdots,\gamma q^{\omega_{\gamma}-1}(\mathrm{mod} \ \frac{\omega_{\gamma}n}{\tau})\}.$$
\end{theorem}

In the following we exhibit the algorithm with an example.

\begin{examples}
	Let $q = 5$ and $n = 3888 = 2^{4}\cdot 3^{5}$. Now we compute the leader of the cosets $c_{3888/5}(2187)$ and $c_{3888/5}(1001)$. Example $6.1.$ in \cite{Zhu2} gives the representatives and the sizes of all $q$-cyclotomic cosets modulo $n$. In particular, we have $|c_{3888/5}(2187)| =4$ and $|c_{3888/5}(1001)| =324$.
	
	For $\gamma=2187$, $n_{\gamma} = 2^{4} = 16$, therefore $\omega_{\gamma}=1$ and $c_{3888/5}(2187)$ is an equal-difference coset. By Lemma \ref{lem 8} the leader of $c_{3888/5}(2187)$ is 
	$$2187(\mathrm{mod} \ \frac{3888}{4})=243.$$
	
	For $\gamma=1001$, $n_{\gamma} = 3888$, therefore $\omega_{\gamma}=2$. By Theorem \ref{thm 6} the leader of $c_{3888/5}(1001)$ is
	$$\mathrm{min}\{1001(\mathrm{mod} \ \frac{2\cdot3888}{324}), 5005(\mathrm{mod} \ \frac{2\cdot3888}{324})\} = \mathrm{min}\{17,13\} =13.$$
\end{examples}

\section*{Acknowledgment}
This work was supported by Natural Science Foundation of Beijing Municipal(M23017). 

\section*{Data availability}
Data sharing not applicable to this article as no data sets were generated or analysed during the current study.

\section*{Declaration of competing interest}
The authors declare that we have no known competing financial interests or personal relationships that 
could have appeared to influence the work reported in this paper.

\end{document}